\documentclass[12pt]{article}

\usepackage{amsthm,amsmath,stmaryrd,bbm,geometry,color}
\usepackage{amssymb}
\usepackage[english]{babel}
\usepackage[utf8]{inputenc}
\usepackage{graphicx}
\usepackage{verbatim}
\usepackage{enumitem}
\usepackage{mathtools}
\usepackage[titletoc]{appendix}
\usepackage{bm}
\usepackage{textcomp}
\usepackage{textcomp}
\usepackage{authblk}
\usepackage[authoryear]{natbib}

\usepackage[dvipsnames]{xcolor}
\usepackage[hyperindex=true,frenchlinks=true,colorlinks=true,
citecolor=Mahogany,linkcolor=Red,urlcolor=Tan,linktocpage]{hyperref}

\usepackage{tikz}
\usetikzlibrary{shapes}
\usetikzlibrary{fit}
\usetikzlibrary{decorations.pathmorphing}

\title{An exponential inequality for Hilbert-valued
$U$-statistics of i.i.d.\ data}

\date{\today}
\author{Davide Giraudo}
\affil[$\dagger$]{Institut de Recherche Mathématique Avancée
UMR 7501, Université de Strasbourg and CNRS
7 rue René Descartes
67000 Strasbourg, France}

\numberwithin{equation}{section}
\renewcommand{\leq}{\leqslant}
\renewcommand{\geq}{\geqslant}

\newtheorem{Theorem}{Theorem}[section]
\newtheorem{Proposition}[Theorem]{Proposition}
\newtheorem{Lemma}[Theorem]{Lemma}
\newtheorem{Definition}[Theorem]{Definition}
\newtheorem{Corollary}[Theorem]{Corollary}

\theoremstyle{remark}

\tikzstyle{Vertex}=[circle,draw=LimeGreen!80,fill=LimeGreen!8,
inner sep=1pt,minimum size=2mm,line width=1pt,font=\scriptsize]
\tikzstyle{Node}=[Vertex,draw=RoyalBlue!80,fill=RoyalBlue!8,inner sep=1.5pt]
\tikzstyle{Leaf}=[rectangle,draw=Black!70,fill=Black!16,
inner sep=0pt,minimum size=1mm,line width=1.25pt]
\tikzstyle{Edge}=[Maroon!80,cap=round,line width=1pt]
\tikzstyle{Mark1}=[draw=BrickRed!80,fill=BrickRed!8]
\tikzstyle{Mark2}=[draw=BurntOrange!80,fill=BurntOrange!8]
\tikzstyle{EdgeRew}=[->,RedOrange!80,cap=round,thick]

\newcommand{\intent}[1]{\llbracket #1\rrbracket}

\newcommand{\Bca}{\mathcal{B}}

\newcommand{\Fca}{\mathcal{F}}
\newcommand{\Gca}{\mathcal{G}}
\newcommand{\Hca}{\mathcal{H}}

\newcommand{\Sca}{\mathcal{S}}

\newcommand{\B}{\mathbb{B}}
\newcommand \ens[1]{\left\{ #1\right\}}

\newcommand \R{\mathbb R}

\newcommand{\Hi}{\mathbb{H}}
\newcommand \N{\mathbb N}
\newcommand \PP{\mathbb P}
\newcommand{\el}{\mathbb L}
\newcommand{\E}[1]{\mathbb E\left[#1\right]}

\newcommand \Z{\mathbb Z}

\newcommand \abs[1]{\left|#1\right|}

\newcommand{\pr}[1]{\left(#1\right)}
\newcommand{\norm}[1]{\left\lVert #1 \right\rVert}
\newcommand{\gr}[1]{\bm{#1}}

\newcommand{\gri}{\gr{i}}
\newcommand{\grj}{\gr{j}}
\newcommand{\conv}{\underset{\operatorname{conv}}{\leq}}

\newcommand{\inc}{\operatorname{Inc}}
\newcommand{\ind}[1]{\mathbf{1}_{#1}}

\newcommand{\dec}{\operatorname{dec}}
\newcommand{\scal}[2]{\left\langle #1,#2\right\rangle}
\begin{document}


\maketitle

\begin{abstract}
In this paper, we establish an exponential inequality for 
$U$-statistics of i.i.d.\ data, varying kernel and taking values in a
separable
Hilbert space. The bound are expressed as a sum of an exponential term
plus
an other one involving the tail of a sum of squared norms.  We start
by the
degenerate case. Then we provide applications to $U$-statistics 
of not necessarily degenerate fixed kernel, weighted $U$-statistics
and
incomplete $U$-statistics.
\end{abstract}

\section{An exponential inequality for Hilbert-valued degenerate $U$-statistics}
\label{sec:ineg_exp_Ustats}
\subsection{Definition of $U$-statistics}
Given an i.i.d.\ sequence $\pr{\xi_i}_{i\geq 1}$ taking values in a measurable 
space $\pr{S,\Sca}$, a separable
Banach space $\pr{\B,\norm{\cdot}_{\B}}$ and measurable functions
$h_{\gri}\colon S^m\to\B$ (where $S^m$ is endowed with the product 
$\sigma$-algebra of $\Sca$), the $U$-statistic of kernels
$\pr{h_{\gri}}$ is defined as
\begin{equation} \label{eq:def_U_stats}
  U_n\pr{\pr{h_{\gri} }_{\gr{i}\in\inc^m }     }
=\sum_{\gri\in\inc^m_n}h_{\gri}
 \pr{\xi_{\gri} },
\end{equation}
where $\inc^m=\ens{\pr{i_k}_{k=1}^m 1\leq
i_1<\dots<i_m}$,
$\inc^m_n=\ens{\gri=\pr{i_k}_{k=1}^m\in\N^m,1\leq
i_1<i_2<\dots<i_m\leq
n}$ and $\xi_{\gri}= \pr{\xi_{i_1},\dots,\xi_{i_m}}$.

When the kernels $h_{\gri}$ are index free, that is, $h_{\gri}=h$ 
and $h\pr{\xi_{\gri}}$ is integrable, it is known that 
$U_n\pr{h}/\binom n2$ converges almost surely to 
$\E{h\pr{\xi_{\intent{1,m}}}}$ (see \cite{MR0026294}), where $\xi_{\intent{1,m}}
=\pr{\xi_1,\dots,\xi_m}$ hence $U_n\pr{h}/\binom n2$ is an 
unbiased estimator of $\E{h\pr{\xi_{\intent{1,m}}}}$. Allowing the kernel $h$ 
to vary with the index opens a wider range of applications, including weighted 
$U$-statistics and also incomplete $U$-statistics.

In this paper, we will consider  $U$-statistics 
taking values in a separable Hilbert space that will be denoted by $\Hi$. We 
provide an inequality for the tail of 
$\norm{U_{n}\pr{h_{\gri}}  }_{\Hi}$. The upper bound is a addition of an 
exponential term with an other one which is expressed in terms 
via a tail of a sum of squared norms of some random variables linked 
with the kernels $h_{\gri}$ and the sequence $\pr{\xi_i}_{i\geq 1}$.

Exponential inequalities for $U$-statistics of i.i.d. data have been obtained 
by \cite{MR1323145}, \cite{MR1336800}, \cite{MR1153808}, \cite{MR1130366}, 
\cite{MR1857312}, \cite{MR2073426} for bounded kernels, \cite{MR1655931} for 
kernels having finite exponential moments. 

For dependent sequences, the following results are available for bounded 
kernels. In the mixing case, 
a deviation inequality has been established by \cite{MR3769824} for fixed 
kernel 
under an assumption on $\beta$-mixing coefficients, and conditional 
$\alpha$-mixing 
coefficients. \cite{MR2655834,MR3943118} dealt with the $\phi$-mixing case. 
The case of $\alpha$-mixing data combined with a control on the decay of the 
Fourier transform of the kernel was treated by \cite{MR4059185}. 
\cite{MR4550210} 
gave exponential inequalities for $U$-statistics of order two, of  varying 
kernel case where data is a function of a Markov chain.

One of the main motivations of studying Hilbert space valued $U$-statistic is 
to treat spatial sign for robust tests (see 
\cite{MR3650400,wegner2023robustchangepointdetectionfunctional,MR4499388}), 
where the sign function is replaced by the
map sending the origin to itself and
$u\mapsto u/\norm{u}_{\Hi}$ if $u\neq 0$. Wilcoxon-Mann-Whitney-type test are 
also studied in
\cite{MR3335109}. 
The consideration of $U$-statistics with varying kernel allows in particular to 
treat weighted $U$-statistics. These ones are widely used for change-point 
tests 
by 
\cite{MR4372100,MR4588210}, ranking problems by \cite{MR3517099,MR2396817}. 
Gini 
mean-differences, that is, when $m=2$, $S=\Hi$ and $h_{i_1,i_2}\pr{u_1,u_2}
=\norm{u_1-u_2}_{\Hi}$, when the sequence $\pr{\xi_i}_{i\geq 1}$ takes values 
in a separable Hilbert space can be useful in the study of 
functional data. The latter has also been done via $U$-statistics 
in \cite{MR4573388}. 
Empirical $U$-statistics have also been studied (see 
\cite{MR1183150} for the i.i.d.\ case, \cite{MR1621734,MR1455799,MR1261803} in 
the mixing case). They can be viewed as random elements of an $\el^2$ space 
hence the deviation inequalities we established can be used.

Our results deal with Hilbert space valued $U$-statistics, which are not the 
most general Banach spaces where $U$-statistics are involved 
(see for instance \cite{MR2294982} and 
\cite{giraudo2024}, where type $2$, 
respectively, smooth Banach spaces are considered). Let us explain why our work 
does not address the latter case. Our argument rests on a deviation inequality 
for Hilbert valued martingales. Using the dimension reduction 
explained in \cite{MR3077911}, section 4, 
such an inequality could be established for smooth Banach spaces (like in 
\cite{MR4414404}), at the cost 
of a constant that depends on $\B$ through a constant $C$ such that for each 
$x,y\in\B$, 
\begin{equation}\label{eq:norm_p_lisse}
 \norm{x+y}_{\B}^p\leq\norm{x}_{\B}^p+ pJ_x\pr{y}+C\norm{y}_{\B}^p,
\end{equation}
where $1<p\leq 2$ and $J_x\colon\B\to\R$ is linear and continuous. We would like 
to use an induction argument 
and a change of Banach space, that is, taking $\B^N$ like in \cite{MR4754186} 
but the constant $C$ 
in \eqref{eq:norm_p_lisse} may depend on $N$ (this was not an issue in the 
aforementioned paper, since the involved constant  
is related to a martingale moment inequality and is the same for $\B$ and 
$\B^N$). 

\subsection{An exponential inequality for degenerate $U$-statistics}
Our approach to derive deviation inequalities rests on martingale 
methods. In general, $U$-statistic of the form
\eqref{eq:def_U_stats} are not sums of martingale differences in 
some sense. In order to use such properties for each
index of summation, we will need the following assumption.
\begin{Definition}
 The $U$-statistic $U_n\pr{\pr{h_{\gri}}}$ defined as in 
\eqref{eq:def_U_stats} is degenerate 
 if for each $\gri\in\inc^m$ and $j_0\in\intent{1,m}$,   
 \begin{equation}\label{eq:def_degeneree}
\E{h_{\gri}\pr{\xi_{\intent{1,m}}} 
\mid \xi_j,j\in\intent{1,m}\setminus\ens{j_0} }=0\mbox{ a.s.},
 \end{equation}
where $\intent{a,b}=\ens{k\in\N,a\leq k\leq b }$ and 
$\xi_{\intent{1,m}}=\pr{\xi_1,\dots,\xi_m}$.
\end{Definition}
We will see later what can be done if \eqref{eq:def_degeneree} 
does not hold.
The inequality for degenerate $U$-statistics reads a follows.

\begin{Theorem}\label{thm:ineg_exp_deg_Ustat}
 Let $m\geq 1$. There exist constants $A_m$, $B_m$ and $C_m$ such that if 
$\pr{\Hi,\scal{\cdot}{\cdot}}$ is a separable Hilbert space,  
$\pr{\xi_i}_{i\geq 1}$ is an i.i.d.\ sequence taking values in a measurable 
space 
$\pr{S,\Sca}$,
$h_{\gri}\colon S^m\to \Hi$ are measurable functions for which  
\eqref{eq:def_degeneree} holds, then 
\begin{multline}\label{eq:inequalite_exp_deg_Ustats}
\PP\pr{ \max_{m\leq n\leq N}\norm{ 
\sum_{\gri\in\inc^m_n}h_{\gri}
 \pr{\xi_{\gri}  }}_{\Hi}  >x}\leq A_{m}\exp\pr{-\pr{\frac 
xy}^{\frac 2m}}
\\+B_{m}\int_1^\infty u\pr{1+ \log u  }^{\frac{m\pr{m+1}}2-1}
\PP\pr{\sqrt{\sum_{\gri\in\inc^m_N}\norm{h_{\gri}
 \pr{ \xi_{\gri}^{\dec} } }_{\Hi}^2 } >C_myu   
}du,
\end{multline}
where 
\begin{equation*}
 \xi_{\gri}^{\dec}= \pr{\xi_{i_1}^{\pr{1}},\dots,\xi_{i_m}^{\pr{m}} }
\end{equation*}
and $\pr{\xi_i^{\pr{\ell}}}_{i\geq 1},\ell\in\intent{1,m}$, are 
independent copies of 
$\pr{\xi_i}_{i\geq 1}$.
\end{Theorem}
Let us make some comments on Theorem~\ref{thm:ineg_exp_deg_Ustat}. 
 
 When $M_{\gri}:= \sup_{s_1,\dots,s_m\in 
S}\norm{h_{\gri}\pr{s_1,\dots,s_m} }_{\Hi}$ 
 is finite, the right hand side of \eqref{eq:inequalite_exp_deg_Ustats} 
 vanishes with the choice 
$y=\sqrt{\sum_{\gri\in\inc^m_N}M_{\gri}^2} /C_m$ 
hence we get 
\begin{equation}
 \PP\pr{ \max_{m\leq n\leq N}\norm{ 
\sum_{\gri\in\inc^m_n}h_{\gri}
 \pr{\xi_{\gri}  }}_{\Hi}  >x} 
 \leq A_{m}\exp\pr{-x^{\frac 2m}C_m^{\frac 2m} 
\pr{\sum_{\gri\in\inc^m_N}M_{\gri}^2}^{-m}   }.
\end{equation}
It is also known by \cite{MR1323145} that when $h=h_{\gri}$ bounded and  
degenerate with respect to $\pr{\xi_i}_{i\geq 1}$, the exponent
$2/m$ in the exponential term is optimal.

Notice also that the right hand side of
\eqref{eq:inequalite_exp_deg_Ustats} is
finite if and only if for each $\gri\in\inc^m_n$,
\begin{equation*}
\E{\norm{h_{\gri}\pr{\xi_{\intent{1,m}}}}_{\Hi}^2\pr{1+\log\pr{\norm{
h_{\gri}\pr{\xi_{\intent{1,m}}}}_{\Hi}}}^{\frac{m\pr{m+1}}{2} -1}
}<\infty.
\end{equation*}
In particular, this require a bit more than
the existence of a finite second moment. This
is not a restriction at all, since in the application  cases, the
random variable $\norm{h_{\gri}\pr{\xi_{\intent{1,m}}}}_{\Hi}$ is
supposed to admit finite exponential moments.
When we only have finite moment of order   between $1$ and $2$, one
can use the deviation inequality presented
in \cite{giraudo2024}.

Let us explain how Theorem~\ref{thm:ineg_exp_deg_Ustat} can be applied. The 
value of $x$ will be imposed by the application we have in mind. Moreover, for 
a fixed $x$,  the  right hand side of \eqref{eq:inequalite_exp_deg_Ustats} 
is the sum of two terms, an exponential one, which is increasing in $y$ and 
and a tail one, which is non-increasing in $y$. Usually, we will 
choose $y$ depending on $N$ and $x$ such that $\exp\pr{-\pr{x/y}^{2/m}}$ has 
the wanted decay, which will furnish a condition on the tail of 
$\sum_{\gri\in\inc^m_N}\norm{h_{\gri}
 \pr{\xi_{\gri}^{\dec} } }_{\Hi}^2$. However, the 
tail of such random variable is not easy to control. 
 We will provide some examples where it will be possible to bound it.

\section{Applications}

The result of Theorem~\ref{thm:ineg_exp_deg_Ustat} apply 
when the kernels are degenerate. In the next examples, we will 
see which kind of bound can be obtained for non-necessarily 
degenerate kernels, and also with a bound involving 
the random variables $\xi_{\intent{1,m}}$ instead of 
$\xi_i$, $1\leq i\leq N$. 

\subsection{Index free kernel}

In this Subsection, we consider the case where the functions 
$h_{\gri}$ 
do not depend on $\gri$ and will be simply denoted as $h$. The 
corresponding $U$-statistic will be expressed as 
\begin{equation*}
U_{m,n}\pr{h}:=\sum_{\gri\in\inc^m_n}
h \pr{\xi_{\gri}}.
\end{equation*}
The degeneracy assumption \eqref{eq:def_degeneree} reads as 
for each $j_0\in\intent{1,m}$, 
\begin{equation*} 
  \E{ h\pr{\xi_{\intent{1,m}}}
\mid \xi_{\intent{1,m}\setminus\ens{j_0}}} =0\mbox{ a.s.}
\end{equation*}
Of course, such an assumption has no reason to take place in general. We thus 
need to decompose $h$ into a sum of functions 
of less than $m$-variables for which the corresponding $U$-statistic is 
degenerate.

\begin{Definition}
Let $\pr{S,\Sca}$ be a measurable space and let $\pr{\Hi,\scal{\cdot}{\cdot}}$ 
be a separable Hilbert space.
We say that the kernel $h\colon S^m\to\Hi$ is symmetric if for each 
$s_1,\dots,s_m\in\Hi$ and each bijective map $\sigma\colon\intent{1,m}\to 
\intent{1,m}$, 
$h\pr{s_{\sigma\pr{1}},\dots,s_{\sigma\pr{m}}}=h\pr{s_1,\dots,s_m}$.
\end{Definition}

Define for $k\in\intent{1,m}$ the kernels
\begin{equation}\label{eq:def_hk}
h_k\colon S^k\to\Hi, h_k\pr{s_1,\dots,s_k}=
\sum_{j=0}^k\pr{-1}^{k-j}\sum_{\pr{u_\ell}_{\ell=1}^j \in\inc^j_k}
\E{h\pr{s_{u_1},\dots,s_{u_j},\xi_{j+1},\dots,\xi_{m}}}
\end{equation}
and $h_0=\E{h\pr{\xi_1,\dots,\xi_m}}$.
In this way, the following equality, known as the Hoeffding's decomposition 
(see \cite{MR0026294}), 
takes place:
\begin{equation}\label{eq:decomposition_Hoeffding}
U_{m,n}\pr{h}=\sum_{k=0}^m\binom{m}k\frac{\binom{n}{m}}{\binom{n}{k}}
U_{k,n}\pr{h_k}.
\end{equation}
When $h$ is degenerate, the terms of index $k\in\intent{0,m-1}$ vanish. 
Also, observe that for each $k$, $U_{k,n}\pr{h_k}$ is a degenerate 
$U$-statistic and that if $\E{h\pr{\xi_{\intent{1,m}}\mid \xi_{ 
\intent{1,k}}}}=0$, 
then $h_k\pr{\xi_{\intent{1,k}}}=0$ almost surely. 
\begin{Definition}\label{def:deg_ordre_d}
Let $m\geq 1$, let $\pr{S,\Sca}$ be a measurable space, let $\pr{\xi_i}_{i\geq 
1}$ be an i.i.d.\ sequence taking values $S$ and let 
$\pr{\Hi,\scal{\cdot}{\cdot}}$ be a separable Hilbert space. We say that 
the kernel $h\colon S^m\to\Hi$ is degenerate of order $d$ if for each 
$k\in\intent{0,d-1}$, $U_{k,n}\pr{h_k}=0$, where $h_k$ is defined as in 
\eqref{eq:def_hk}. 
\end{Definition}

It is clear that in view of \eqref{eq:decomposition_Hoeffding} that if 
$h$ is degenerate of order $d$, then 
\begin{equation}\label{eq:decomposition_Hoeffding_deg_d}
U_{m,n}\pr{h}=\sum_{k=d}^m\binom{m}k\frac{\binom{n}{m}}{\binom{n}{k}}
U_{k,n}\pr{h_k}.
\end{equation}

\begin{Theorem}\label{thm:dev_ineg_fixed_kernel}
Let $m\geq 1$. There exists constants $A_m$, $B_m$ and $C_m$ such that 
if $\pr{\Hi,\scal{\cdot}{\cdot}}$ is a separable Hilbert space, 
$\pr{\xi_i}_{i\geq 1}$ is an i.i.d.\ sequence taking values in a 
measurable space $\pr{S,\Sca}$, $h\colon S^m\to\Hi$ is symmetric 
and degenerate of order $d$, then for each positive $x$ and $y$, 
\begin{multline}\label{eq:dev_ineg_fixed_kernel}
\PP\pr{\max_{m\leq n\leq N}\norm{U_{m,n}\pr{h}}_{\Hi}>xN^{m-\frac d2}}
\\
\leq A_m\exp\pr{-\pr{\frac xy}^{\frac 2d}}
+B_m\sum_{k=d}^m\int_1^\infty\PP\pr{H_k>C_muN^{\frac{k-d}2}x^{1-\frac 
kd}y^{\frac kd}}u
\pr{1+\log u}^{\frac{m\pr{m+1}}{2}}du,
\end{multline}
where
\begin{equation*}
H_k:= \E{\norm{h\pr{\xi_{\intent{1,m } }}}  _{\Hi}  \mid 
\xi_{\intent{1,k}}} .
\end{equation*}
\end{Theorem}
This result relates to  Corollary~1.1 in \cite{MR4294337} in the following way. 
Our result deals with Hilbert-valued $U$-statistics, while that in 
\cite{MR4294337} is restricted to the real-valued case. However, in this latter 
case, the constants 
are explicit which is not the case in \eqref{eq:dev_ineg_fixed_kernel}.

Rates of convergence in the law of large numbers can be derived 
in exactly the same way as it was done in \cite{MR4294337}, 
replacing the absolute value by the Hilbert norm. We will 
therefore simply state them, without proof.

 \begin{Corollary}\label{cor:large_deviation_Ustats}
Let $\pr{S,\Sca}$ be a measurable space, $m\geq 2$,  $h\colon S^m\to\R$ be a 
symmetric function,  and 
 let $\pr{\xi_i}_{i\geq 1}$ be an i.i.d.\ sequence of random variables with 
values in $S$.
 Suppose that $h$ is degenerate of order $d$.  If there exists a positive 
 $\gamma>0$ such that $M:=\sup_{t>0}\exp\pr{t^\gamma} 
 \PP\pr{   \norm{h\pr{\xi_{\intent{1,m}}}}_{\Hi}>t }$ is finite,  then for all 
$x>0$, 
 \begin{equation*}
 \PP\pr{\max_{m\leq n\leq N}\norm{U_{m,n}\pr{h} }_{\Hi}> N^{m}x}\leq K_1
 \exp\pr{ -K_2 N^{\frac{d\gamma}{2+d\gamma}}x^{\frac{2\gamma }{2+d\gamma} }   },
 \end{equation*}
 where $K_1$ and $K_2$ depend on $\gamma$, $M$ and $m$. 
 \end{Corollary}

\subsection{Weighted sums}

In this Subsection, we provide results for weighted sums 
of non-necessarily degenerate $U$-statistics. 

When $h_{\gri}$ does not depend on $\gri$ 
or more generally when $h_{\gri}=T_{\gri}\pr{h}$, where 
$T_{\gri}\colon\Hi\to\Hi$ is a bounded linear operator and $h$ is degenerate, 
it is possible 
to control the tail of 
$\sum_{\gri\in\inc^m_N}\norm{h_{\gri}
 \pr{\xi_{\gri}^{\dec}  } }_{\Hi}^2$ in terms 
 of the operator norms of $T_{\gri}$, defined as 
$\norm{T_{\gri}}_{\mathcal{B}\pr{\Hi}}=\sup\ens{\norm{T_{\gri}\pr{h}}_{\Hi},h\in
\Hi ,
\norm{h}_{\Hi }=1 } $, and the tail of 
 $\norm{h\pr{\xi_{\intent{1,m}} }}_{\Hi}$.  

When $h$ is not necessarily degenerate, one can still do the Hoeffding's 
decomposition and get that 
\begin{equation}\label{eq:Hoeffding_weighted}
\sum_{\gri\in\inc^m_n}T_{\gri}\pr{
 h\pr{\xi_{\gri  } }}
 =\sum_{k=d }^m 
 \sum_{\gri\in\inc^k_n}
 a_{\gri}^{\pr{n,k}}\pr{
 h_k\pr{\xi_{\gri}   }},
\end{equation}
where $h_k$ is defined as in \eqref{eq:def_hk} and for 
$\gri=\pr{i_1,\dots,i_k} \in\inc^k_n$, 
\begin{equation*} 
a_{\gri}^{\pr{n,k}}=\sum_{\gr{j}=\pr{j_1,\dots,j_m}\in\inc^m_n,\ens{i_1,\dots,
i_k } 
\subset\ens{j_1,\dots,j_m}  } T_{\grj}.
\end{equation*}
When $m=2$, one has 
\begin{equation*}
a_i^{\pr{n,1}}=\sum_{j=1}^n\pr{T_{i,j}+T_{j,i}},
\end{equation*}
with the convention that $T_{j,i}=0$ if $j\leq i$. For $m=3$, one has 
\begin{equation*}
a_{i_1,i_2}^{\pr{n,2}}=\sum_{j=1}^n
\pr{T_{i_1,i_2,j}+T_{i_1,j,i_2}+T_{j,i_1,i_2}}
\end{equation*}
\begin{equation*}
a_{i_1}^{\pr{n,1}}=\sum_{1\leq j_1<j_2\leq n}
\pr{T_{i_1,j_1,j_2}+
T_{j_1,i_1,j_2}+T_{j_1,j_2,i_1}},
\end{equation*}
with the convention that $T_{k_1,k_2,k_3}=0$ if $k_1<k_2<k_3$ does 
not hold.

Since the operators $a_{\gri}^{\pr{n,k}}$ depend on $n$, 
we cannot formulate a deviation inequality for the running 
maximum.
A deviation inequality for 
weighted and non-necessarily degenerate $U$-statistics reads as follows. 
 
\begin{Corollary}\label{cor:weighted_Ustat_non_deg}
Let $m\geq 1$.  There exist constants $a_m$, $b_m$ and $c_m$ such that if 
$\pr{\Hi,\scal{\cdot}{\cdot}}$ is a separable Hilbert space,  
$\pr{\xi_i}_{i\geq 1}$ is an i.i.d.\ sequence taking values in a measurable 
space 
$\pr{S,\Sca}$,
$h \colon S^m\to \Hi$ is a measurable function which is degenerate of 
order $d$ and $T_{\gri}\colon \Hi\to \Hi$ are bounded linear 
operators, 
then for each positive $x$ and $y$,
\begin{multline}\label{eq:inequalite_exp_deg_Ustats_weighted}
\PP\pr{  \norm{ 
\sum_{\gri\in\inc^m_n}T_{\gri}\pr{
 h\pr{\xi_{\gri}  }}   }_{\Hi}  >xn^{\frac{m-d}2} 
\sqrt{\sum_{\gri\in\inc^m_n}\norm{T_{\gri}}_{\Bca\pr{\Hi}}^2   } }\leq 
a_{m}\exp\pr{-\pr{\frac 
xy}^{\frac 2m}}
\\+b_{m}\int_1^\infty u\pr{1+ \log u  }^{\frac{m\pr{m+1}}2}
\PP\pr{ H>c_myu   
}du,
\end{multline}
where 
\begin{equation*}
 H=\norm{h\pr{\xi_{\intent{1,m}}}}_{\Hi}.
\end{equation*}

\end{Corollary} 
 
\subsection{Incomplete $U$-statistics} 
\label{subsec:incomplete}
The computation of a $U$-statistic of order $m$ based on a sample 
of size $n$ requires a computation of a number of terms of order
$n^m$, which can be high in practise. For
this reason, \cite{MR0474582} introduced the so-called incomplete
$U$-statistics, whose rough idea is to put a random weight
equal to $0$ or $1$, reducing the number of terms that one has 
to compute. Here are the most common ways of defining the 
random weights.
\begin{itemize}
 \item Sampling without replacement: we pick without replacement $N$ $m$-uples
which belong to $\inc^m$.
\item Sampling with replacement:  we pick with replacement $N$ $m$-uples
which belong to $\inc^m$.
\item Bernoulli sampling: consider for
each $\gri\in\inc^m_n$, a random
variable $Z_{n;\gri}$ taking the value $1$ with probability $p_n$ and
$0$ with probability $1-p_n$. Moreover, we assume that the family
$\pr{Z_{n,\gri}}_{\gri\in\inc^m_n}$ is independent and also
independent of the sequence $\pr{\xi_i}_{i\in\Z}$.
\end{itemize}
The weak convergence of incomplete $U$-statistics 
has been established in \cite{MR0753810}.
Rates in the law of large numbers were obtained by \cite{MR2915089} and for 
$\el^1$ convergence by \cite{MR4503429}.
Recent papers include incomplete $U$-statistics in a high dimensional setting, 
like 
\cite{MR4025737}, and \cite{MR4289844} obtained a central
limit theorem when the data is based on a triangular array.

Our first result deals with sampling with or without replacement. 

\begin{Corollary}[Deviation inequality for sampling with and without 
replacement]\label{cor:deviation_incomplete_Ustats_replacement}
Let $m\geq 1$.  There exist constants $a_m$, $b_m$ and $c_m$ such that if 
$\pr{\Hi,\scal{\cdot}{\cdot}}$ is a separable Hilbert space,  
$\pr{\xi_i}_{i\geq 1}$ is an i.i.d.\ sequence taking values in a measurable 
space 
$\pr{S,\Sca}$,
$h \colon S^m\to \Hi$ is a measurable function which is degenerate of 
order $d$, $N\geq 1$ and $n\geq m$ are integers and 
$\pr{Z_{n,\gri}}_{\gri\in\inc^m_n}$ is a collection of random variables taking 
the 
values $0$ or $1$ which is independent of $\pr{\xi_i}_{i\geq 1}$ and such that 
$\sum_{\gri\in\inc^m_n}Z_{n,\gri}
=N$ and $x,y>0$, then 
\begin{multline}\label{eq:deviation_inequality_sampling_rep}
 \PP\pr{\norm{\sum_{\gri\in\inc_n^m}Z_{n,\gri}h\pr{\xi_{\gri}}    
}_{\Hi}>x \sqrt{N}\sqrt{\min\ens{N,n^{m-d}  }}  }\leq 
a_m\exp\pr{-\pr{\frac{x}{y}}^{\frac 2m}}\\
+b_m \int_1^\infty  u\pr{1+\log u}^{\frac{m\pr{m+1}}2}
\PP\pr{ H>c_m yu}du.
\end{multline}
 \end{Corollary}
Observe that the sampling with or without replacement is packed 
gives the same result. The only crucial thing is the numbers of 
random weights that are not zero. 

Also, observe that the normalization $\sqrt{N}\sqrt{\min\ens{N,n^{m-d}  }}$ is 
decreasing with $d$: if the $U$-statistic is degenerate of 
order $m$, we only need a normalization by $\sqrt{N}$ in order 
to get a bounded independent of $N$ and $n$. In the other extreme case, that 
is, when $d=1$, which corresponds to the case where $h_{\gri}\pr{\xi_{\gri}}$ 
is 
centered, if one takes $N\leq n^{m-1}$ (which is logical if one wants to 
reduce significantly the number of elements), then one needs 
the much stronger normalization $N$.

A deviation inequality for Bernoulli sampling reads as follows.
\begin{Corollary}[Deviation inequality for Bernoulli sampling]
\label{cor:deviation_incomplete_Ustats_Bernoulli}
Let $m\geq 1$.  There exist constants $a_m$, $b_m$ and $c_m$ such that if 
$\pr{\Hi,\scal{\cdot}{\cdot}}$ is a separable Hilbert space,  
$\pr{\xi_i}_{i\geq 1}$ is an i.i.d.\ sequence taking values in a measurable 
space 
$\pr{S,\Sca}$,
$h \colon S^m\to \Hi$ is a measurable function which is degenerate of 
order $d$, $N\geq 1$ and $n\geq m$ are integers and 
$\pr{Z_{n,\gri}}_{\gri\in\inc^m_n}$ is a collection of random variables taking 
the 
values $0$ or $1$ with respective probabilities $1-p_n$ and $p_n$, which is 
independent of $\pr{\xi_i}_{i\geq 1}$ and $x,y>0$, then  
\begin{multline}\label{eq:deviation_inequality_sampling_Ber}
 \PP\pr{\norm{\sum_{\gri\in\inc_n^m}Z_{n,\gri}h\pr{\xi_{\gri}}    
}_{\Hi}>x n^{m}\sqrt{p_n}\sqrt{\min\ens{p_n,n^{-d}  }}  }\\
\leq 
a_m\exp\pr{-\frac{n^mp_n^2}2}+ a_m\exp\pr{-\pr{\frac{x}{y}}^{\frac 2m}}
+b_m \int_1^\infty  u\pr{1+\log u}^{\frac{m\pr{m+1}}2}
\PP\pr{H>c_m yu}du.
\end{multline}
\end{Corollary}
Let us compare Corollaries~\ref{cor:deviation_incomplete_Ustats_replacement} 
with
\ref{cor:deviation_incomplete_Ustats_Bernoulli}. The 
right hand side of \eqref{eq:deviation_inequality_sampling_rep}
is similar, up to the constants, to the sum of the second and 
third terms of the right hand side of 
\eqref{eq:deviation_inequality_sampling_Ber}. These quantities are independent 
of $n$ and of the involved parameters of the sampling methods. The 
required normalizations to have a right hand side independent of $n$, $N$ or 
$p_n$
(up the the first term for  \eqref{eq:deviation_inequality_sampling_Ber}) 
depend on $d$ and are the weakest if $d=m$ and the strongest if $d=1$. When 
$N=n^mp_n$, 
the normalizing terms coincide.  
Nevertheless, 
an other exponential term appears in 
\eqref{eq:deviation_inequality_sampling_Ber} due to the control of the 
probability that the sum of the weights is bigger than some number. 

Let us point out that a Bernstein type
exponential inequality for incomplete real valued $U$-statistics has been 
established by \cite{maurer2022exponentialfinitesamplebounds}. The results 
deals with sampling with replacement, bounded kernels and 
the author does not make any assumption on degeneracy. The results 
are expressed in terms of the number of $m$-uples among the chosen $N$ that 
contain a specific value $k$, or that contain both $k$ and $\ell$. Therefore, 
this does not make Corollary~\ref{cor:deviation_incomplete_Ustats_replacement} 
directly comparable with Maurer's results.

If we assume that $H$ has a weak exponential moment, namely that for some 
positive $\kappa_0$ and $\alpha>0$, 
\begin{equation}\label{eq:weak_expo_moment}
\sup_{t>0}\exp\pr{\kappa_0 t^{\alpha}}
\PP\pr{H>t}<\infty,
\end{equation}
the bounds \eqref{eq:deviation_inequality_sampling_rep} and 
\eqref{eq:deviation_inequality_sampling_Ber} can be simplified as follows.
 
For the sampling with or without replacement, one gets the existence of a 
constant $\kappa>0$ such that 
\begin{equation}
\sup_{n\geq m}\sup_{x>0}\exp\pr{\kappa 
x^{\frac{2\gamma}{2+m\gamma}}}\PP\pr{\norm{\sum_{\gri\in\inc_n^m}Z_{n,\gri}h\pr{
\xi_{\gri}}    
}_{\Hi}>x \sqrt{N}\sqrt{\min\ens{N,n^{m-d}  }}  }<\infty.
\end{equation}
For Bernoulli sampling, there exists a constant $K$ such that for each $n$, 
$p_n\in (0,1)$ and $x>0$, 
\begin{equation}
 \exp\pr{-\min\ens{ \frac{n^mp_n}2,\kappa x^{\frac{2\gamma}{2+m\gamma}} 
}}\PP\pr{\norm{\sum_{\gri\in\inc_n^m}Z_{n,\gri}h\pr{\xi_{\gri}}    
}_{\Hi}>x n^{m}\sqrt{p_n}\sqrt{\min\ens{p_n,n^{-d}  }}  }\leq K.
\end{equation}
\section{Proofs}

\subsection{Proof of Theorem~\ref{thm:ineg_exp_deg_Ustat}}
Let us first explain the global idea of proof. It will be show that 
it is sufficient to prove Theorem~\ref{thm:ineg_exp_deg_Ustat} 
for decoupled $U$-statistics and only in the case $x/y>3^{m/2}$. Then we 
proceed 
by induction over $m$. The induction step is performed first via an application 
of the martingale inequality given in 
Proposition~\ref{prop:ineg_deviation_martingale_cas_Hilbert_hyp_var_cond} for 
an appropriated choice of the filtration. Then, in order to use the induction 
assumption, we introduce a new Hilbert space so that we can apply the 
result for $U$-statistics of lower order. 

In order to ease the notation of the exponent of the $\log$ appearing 
in the involved terms in the proof, we define 
\begin{equation*}
 p_m=\frac{m\pr{m+1}}2-1,\quad m\geq 1.
\end{equation*}
In view of Corollary~\ref{cor:decoupling_max}, it suffices to prove 
Theorem~\ref{thm:ineg_exp_deg_Ustat} for decoupled $U$-statistics. Indeed, 
assume that for each $m\geq 1$, we can find constants $A'_m$, $B'_m$ and 
$C'_m$ 
for which the inequality 
\begin{multline*} 
\PP\pr{ \max_{m\leq n\leq N}\norm{ 
\sum_{\gri\in\inc^m_n}h_{\gri}
 \pr{\xi_{\gri}^{\dec}   }}_{\Hi}  >x}\leq 
A'_{m}\exp\pr{-\pr{\frac 
xy}^{\frac 2m}}
\\+B'_{m}\int_1^\infty u\pr{1+ \log u  }^{p_m}
\PP\pr{\sqrt{\sum_{\gri\in\inc^m_N}\norm{h_{\gri}
 \pr{\xi_{\gri}^{\dec}  } }_{\Hi}^2 } >C'_myu   
}du
\end{multline*}
holds for each Hilbert space $\Hi$ and each functions $h_{\gri}$ which 
are degenerate with respect to $\pr{\xi_i}_{i\geq 1}$. Applying 
Corollary~\ref{cor:decoupling_max}, we derive that 
\begin{multline*} 
\PP\pr{ \max_{m\leq n\leq N}\norm{ 
\sum_{\gri\in\inc^m_n}h_{\gri}
 \pr{\xi_{\gri}  }}_{\Hi}  >x}\leq 
A'_{m}K_m\exp\pr{-\pr{\frac{x}{yK_m}}^{\frac 2m}}
\\+B'_{m}K_m\int_1^\infty u\pr{1+ \log u  }^{p_m}
\PP\pr{\sqrt{\sum_{\gri\in\inc^m_N}\norm{h_{\gri}
 \pr{\xi_{\gri}^{\dec}  } }_{\Hi}^2 } >C'_myu   
}du
\end{multline*}
and replacing $y$ by $y/K_m$, we can see that letting $A_m:= A'_mK_m$, 
$B_m=B'_m K_m$ and $C_m=C'_m/K_m$, \eqref{eq:inequalite_exp_deg_Ustats} holds. 

We can also reduce the proof to the case $x/y>3^{m/2}$, by replacing $A_m$ by 
$\max\ens{A_m,e^3}$, so that the contribution of $\max\ens{A_m,e^3} 
\exp\pr{-\pr{x/y}^{2/m}}$ makes the inequality trivial for the range of 
$\pr{x,y}$ for which $x/y\leq 3^{m/2}$.

Therefore, we will prove by induction on $m\geq 1$ the assertion $P\pr{m}$ 
defined as follows: "there exists constants $a_m$, $b_m$ and $c_m$ such that 
for each measurable space $\pr{S,\Sca}$, each independent copies 
$\pr{\xi_i^{k}}$, $k\in\intent{1,m}$ of an i.i.d.\ $S$-valued sequence 
$\pr{\xi_i}_{i\geq 1}$, each separable Hilbert space 
$\pr{\Hi,\scal{\cdot}{\cdot}}_{\Hi}$, each collection $h_{\gri}\colon 
S^m\to \Hi$, $\gri\in\inc^m$ of measurable maps such that for each 
$\gri\in\inc^m$ and each $j_0\in\intent{1,m}$,
\begin{equation}\label{eq:deg_dans_def_Pm}
 \E{h_{\gri}\pr{\xi_{\intent{1,m}  }}\mid 
 \sigma\pr{\xi_j,j\in\intent{1,m}\setminus \ens{j_0}}
 }=0,
\end{equation}
the following inequality holds for each $x,y>0$ satisfying $x/y>3^{m/2}$ and 
each 
$N\geq m$:
\begin{multline*} 
\PP\pr{ \max_{m\leq n\leq N}\norm{ 
\sum_{\gri\in\inc^m_n}h_{\gri}
 \pr{\xi_{\gri}^{\dec}   }}_{\Hi}  >x}\leq 
a_{m}\exp\pr{-\pr{\frac 
xy}^{\frac 2m}}
\\+b_{m}\int_1^\infty u\pr{1+ \log u  }^{p_m}
\PP\pr{\sqrt{\sum_{\gri\in\inc^m_N}\norm{h_{\gri}
 \pr{\xi_{\gri}^{\dec} } }_{\Hi}^2 } >c_myu   
}du."
\end{multline*}
For $m=1$, this follows directly from 
Proposition~\ref{prop:ineg_deviation_sans_var_cond_cas_Hilbert}, since 
\eqref{eq:deg_dans_def_Pm} implies that the considered $U$-statistic 
is a sum of independent centered random variables. \\
Assume that $P\pr{m}$ holds for some $m\geq 1$ and let us show $P\pr{m+1}$.
Let $\pr{S,\Sca}$ be a measurable space, $\pr{\xi_i}_{i\geq 1}$ an i.i.d.\ 
sequence of $S$-valued random variables, let $\pr{\Hi,\scal{\cdot}{\cdot}}$
be a Hilbert space and let $h_{\gri,i_{m+1}}\colon S^{m+1}\to\Hi$ 
be measurable functions such that for each $j_0\in\intent{1,m+1}$, 
 \begin{equation}\label{eq:def_degeneree_m_plus_1}
\E{h_{\gri,i_{m+1}}\pr{\xi_{\intent{1,m+1}} } 
\mid\sigma\pr{\xi_j,j\in\intent{1,m+1}\setminus\ens{j_0} }}=0\mbox{ a.s.}
 \end{equation}
 Let $N\geq m+1$ and $x,y>0$ such that $x/y>3^{\pr{m+1}/2}$.
Let us show that there are constants $a_{m+1}$, $b_{m+1}$ and $c_{m+1}$ 
depending only on $m+1$ such that 
\begin{multline}\label{eq:inequalite_exp_deg_Ustats_m+1}
\PP\pr{ \max_{m+1\leq n\leq N}\norm{ 
\sum_{\gri\in\inc^{m+1}_n}h_{\gri}
 \pr{\xi_{\gri}^{\dec}   }}_{\Hi}  >x}\leq 
a_{m+1}\exp\pr{-\pr{\frac 
xy}^{\frac{2}{m+1}}}
\\+b_{m+1}\int_1^\infty u\pr{1+ \log u  }^{p_{m+1}}
\PP\pr{\sqrt{\sum_{\gri\in\inc^{m+1}_N}\norm{h_{\gri}
 \pr{\xi_{\gri}^{\dec} } }_{\Hi}^2 
} >c_{m+1}yu   }du.
\end{multline}
Define for $j\geq m+1$ the random variable 
\begin{equation*}
D_j:=\sum_{\gri\in\inc^m_{j-1}}
h_{\gri,j}
 \pr{\xi_{\gri}^{\dec} ,\xi_{j}^{\pr{m+1}}  }
\end{equation*} 
 and the $\sigma$-algebras
 \begin{equation*}
  \Fca_m:=\sigma\pr{\xi_{\gri}^{\dec} ,
\gri\in\inc^m_{N-1}  },
 \end{equation*}
\begin{equation*}
 \Fca_j=\Fca_m\vee\sigma\pr{\xi_u^{\pr{m+1}},u\leq j}, 
j\geq m+1.
\end{equation*}
Recall the following elementary property of conditional expectations:
if $Y$ is an integrable random variable on a probability space 
$\pr{\Omega,\Fca,\PP}$ and $\Gca$, $\Hca$ are two sub-$\sigma$-algebras of 
$\Fca$ such that $\Hca$ is independent of $\sigma\pr{Y}\vee\Gca$, then
\begin{equation}\label{eq:elem_prop_cond_exp}
 \E{Y\mid\Gca\vee\Hca}= \E{Y\mid\Gca }.
\end{equation}
Let $j\geq m+1$ be fixed. Then the following equality holds:
\begin{equation*}
 \E{D_j\mid\Fca_{j-1}    }=\sum_{\gri\in\inc^m_{j-1}}
 \E{h_{\gri,j}
 \pr{\xi_{\gri}^{\dec},\xi_{j}^{\pr{m+1}}  }
 \mid\Fca_{j-1}}.
\end{equation*}
Using equality \eqref{eq:elem_prop_cond_exp} for a fixed 
$\gri\in\inc^m_{j-1}$  with 
$Y=h_{\gri,j}
 \pr{\xi_{\gri}^{\dec} ,\xi_{j}^{\pr{m+1}}  }$, 
$\Gca=\sigma\pr{ \xi_{\gri}^{\dec} }$ and 
\begin{equation*}
 \Hca:=\sigma\pr{\xi_{\gr{i'}}^{\dec} ,
\gr{i'}\in\inc^m_{N-1}\setminus\ens{\gri} 
}\vee\sigma\pr{\xi_u^{\pr{m+1}},u\leq j-1},
\end{equation*}
furnishes the equality 
\begin{equation*}
 \E{D_j\mid\Fca_{j-1}    }=\sum_{\gri \in\inc^m_{j-1}}
 \E{h_{\gri,j}
 \pr{\xi_{\gri}^{\dec} ,\xi_{j}^{\pr{m+1}}  }
 \mid \sigma\pr{ \xi_{\gri}^{\dec  }   }},
\end{equation*}
which equals $0$ by \eqref{eq:def_degeneree_m_plus_1}
hence $\pr{D_j}_{j\geq m+1}$ is a martingale difference sequence 
with respect to the filtration $\pr{\Fca_j}_{j\geq m}$. Moreover, 
observe that 
\begin{equation*}
 \E{\norm{D_j}_{\Hi}^2\mid\Fca_{j-1}}
=\E{\norm{D_j}_{\Hi}^2\mid\Fca_{m} \vee  
\sigma\pr{\xi_u^{\pr{m+1}},u\leq j-1}
}
\end{equation*}
hence an other application of \eqref{eq:elem_prop_cond_exp} to 
$\Gca=\Fca_m$ and $\Hca=\sigma\pr{\xi_u^{\pr{m+1}},u\leq j-1}$ gives 
$ \E{\norm{D_j}_{\Hi}^2\mid\Fca_{j-1}}= \E{\norm{D_j}_{\Hi}^2\mid\Fca_{m}}$. 
Using Proposition~\ref{prop:ineg_deviation_martingale_cas_Hilbert_hyp_var_cond} 
with $x_1=x $ and $y_1  =y^{\frac{1}{m+1}}x^{\frac{m}{m+1}}$
gives
  \begin{multline} 
\PP\pr{ \max_{m+1\leq n\leq N}\norm{ 
\sum_{\gri\in\inc^{m+1}_n}h_{\gri}
 \pr{\xi_{\gri}^{\dec}  }}_{\Hi}  >x}=\PP\pr{ \max_{m+1\leq n\leq N}
 \norm{\sum_{j=m+1}^n D_j}_{\Hi}  >x}\\ 
 \leq 4  \exp\pr{-\pr{\frac{x}{y}}^{\frac{2}{m+1}}}
 +4\int_1^\infty u\PP\pr{\sqrt{\sum_{j=m+1}^N\norm{D_j}_{\Hi}^2} >   
 y^{\frac{1}{m+1}}x^{\frac{m}{m+1}}\frac{u}{8} }du,
\end{multline} 
which can be rewritten, by independence between 
$\pr{\xi_{\gri}^{\dec}}_{\gri\in\inc^m}$ and $\pr{\xi_j^{\pr{m+1}}}_{j=m+1}^N$, 
as 
\begin{multline}\label{eq:inequalite_exp_deg_Ustats_m+1_etape_2}
 \PP\pr{ \max_{m+1\leq n\leq N}\norm{ 
\sum_{\gri\in\inc^{m+1}_n}h_{\gri}
 \pr{\xi_{\gri}^{\dec}   }}_{\Hi}  >x}
 \leq    4\exp\pr{-\pr{\frac{x}{y}}^{\frac{2}{m+1}}}\\ 
 + 4\int_1^\infty u\int_{S^{N-m}  
}\PP\pr{  \sum_{j=m+1}^N\norm{\sum_{\gri\in\inc^m_{j-1}}
h_{\gri,j}
 \pr{\xi_{\gri}^{\dec},s_j}  
}_{\Hi}^2 >  
 y^{\frac{2}{m+1}}x^{\frac{2m}{m+1}}\frac{u^2}{64} }
d\PP_{\xi_{\intent{m+1,N}}^{\pr{m+1}}}\pr{s_{\intent{m+1,N}}} du,
\end{multline} 
where for a function $g\colon S^{N-m}\to\R$, 
\begin{equation*}
 \int 
g\pr{s_{m+1},\dots,s_N}d\PP_{\xi_{\intent{m+1,N}}^{\pr{m+1}}}\pr{s_{\intent{m+1,
N}}} =\int g\pr{s_{m+1},\dots,s_N}d\PP_{\xi_{m+1}}\pr{s_{m+1}}
\dots d\PP_{\xi_{N}}\pr{s_{N}}.
\end{equation*}
 In order to control the tail involved in the right hand side of the previous 
equation, we 
 introduce the separable Hilbert space 
$\pr{\widetilde{\Hi},\scal{\cdot}{\cdot}_{\widetilde{\Hi}}}$ which is defined 
as $\widetilde{\Hi}=\ens{\pr{x_j}_{j\in\intent{m+1,N}},x_j\in\Hi}$ and 
\begin{equation*}
\scal{\pr{x_j}_{j\in\intent{m+1,N}}}{\pr{y_j}_{j\in\intent{m+1,N}}}_{\widetilde{
\Hi}}=\sum_{j=m+1}^N\scal{x_j}{y_j}_{\Hi},
\end{equation*}
and for fixed $s_{j},j\in\intent{m+1,N}$, we define the functions 
$\widetilde{h_{\gri; \pr{s_{j}}_{j\in\intent{1,m}}}}\colon S^m\to 
\widetilde{\Hi}
$ by 
\begin{equation}\label{eq:def_til_h_i}
 \widetilde{h_{\gri; 
\pr{s_{j}}_{j\in\intent{m+1,N}}}}\pr{s_1,\dots,s_m}
=\pr{h_{\gri,j}
 \pr{s_1,\dots,s_m,s_{j}}\ind{\gri\in\inc^m_{j-1} }    
}_{j\in\intent{m+1,N}}.
\end{equation}
Consequently, \eqref{eq:inequalite_exp_deg_Ustats_m+1_etape_2} can be 
rephrased as 
\begin{multline}\label{eq:inequalite_exp_deg_Ustats_m+1_etape_3}
\PP\pr{ \max_{m+1\leq n\leq N}\norm{ 
\sum_{\gri\in\inc^{m+1}_n}h_{\gri}
 \pr{\xi_{\gri}^{\dec}  }}_{\Hi}  >x}
 \leq 4   \exp\pr{-\pr{\frac{x}{y}}^{\frac{2}{m+1}}}\\ 
 +4\int_1^\infty u\int_{S^{N-m}  
}\PP\pr{  \norm{\sum_{\gri\in\inc^m_{N-1}}
 \widetilde{h_{\gri; 
\pr{s_{j }}_{j\in\intent{1,m}}}}
 \pr{\xi_{\gri}^{\dec}  }  
}_{\widetilde{\Hi}} >  
 y^{\frac{1}{m+1}}x^{\frac{m}{m+1}}\frac{u}{8} 
}d\PP_{\xi_{\intent{m+1,N}}^{\pr{m+1}}}\pr{s_{\intent{m+1,N}}} du.
\end{multline}
By assumption \eqref{eq:def_degeneree_m_plus_1}, the following 
equality takes place for $\PP_{\xi_m^{\pr{m+1}}}\otimes \cdots\otimes
\PP_{\xi_N^{\pr{m+1}}}$ almost every $\pr{s_{j}}_{j\in\intent{m+1,N}}$:
\begin{equation*}
\E{ \widetilde{h_{\gri; 
\pr{s_{j}}_{j\in\intent{m+1,N}}}}  
\pr{\xi_{\intent{1,m}}}\mid \xi_{\intent{1,m}\setminus \ens{j_0} 
} }=0.
\end{equation*}
Therefore, an application of $P\pr{m}$ for such $s_{j}$, fixed $u\geq 1$ and 
$x$ 
and $y$ 
replaced respectively by 
\begin{equation*}
 x_2=y^{\frac{1}{m+1}}x^{\frac{m}{m+1}}\frac{u}{8}
, \quad 
 y_2=\frac{yu }{8\pr{1+\log u}^{\frac m2}},
\end{equation*}
gives 
\begin{multline}\label{eq:inequalite_exp_deg_Ustats_m+1_etape_4}
\PP\pr{ \max_{m+1\leq n\leq N}\norm{ 
\sum_{\gri\in\inc^{m+1}_n}h_{\gri}
 \pr{\xi_{\gri}^{\dec}}   }_{\Hi}  >x}\leq 4   
\exp\pr{-\pr{\frac{x}{y}}^{\frac{2}{m+1}}}
\\ 
 +2 a_m\int_1^\infty u
 \exp\pr{-\pr{\frac xy}^{\frac{2}{m+1}}\pr{1+\log u}  } du
+b_m \int_{S^{N-m}}  \int_1^\infty  \int_1^\infty uv\pr{1+\log 
v}^{p_m} \\
 \PP\pr{ \sqrt{ \sum_{\gri\in\inc^m_{N-1}}
 \norm{\widetilde{h_{\gri; 
\pr{s_{j }}_{j\in\intent{m+1,N}}}}
 \pr{\xi_{\gri}^{\dec }  }
}_{\widetilde{\Hi}}^2} >\frac{yu vc_m}{8\pr{1+\log u}^{\frac m2}}    }  
dudv d\PP_{\xi_{\intent{m+1,N}}^{\pr{m+1}}}\pr{s_{\intent{m+1,N}}}.
\end{multline} 
Since $x/y>3^{\pr{m+1}/2}$,  we get 
\begin{equation*}
2 a_m\int_1^\infty u
 \exp\pr{-\pr{\frac xy}^{\frac{2}{m+1}}\pr{1+\log u}  } du
 \leq 2a_m
\exp\pr{-\pr{\frac xy}^{\frac{2}{m+1}}} . 
\end{equation*}
Moreover, 
using the expression of $\norm{\cdot}_{\widetilde{\Hi}}$ and  
\eqref{eq:def_til_h_i} allows us to derive from  
\eqref{eq:inequalite_exp_deg_Ustats_m+1_etape_4} that   
\begin{multline}\label{eq:inequalite_exp_deg_Ustats_m+1_etape_5}
\PP\pr{ \max_{m+1\leq n\leq N}\norm{ 
\sum_{\gri\in\inc^{m+1}_n}h_{\gri}
 \pr{\xi_{\gri}^{\dec}  }}_{\Hi}  >x}\leq 
 \pr{4+2a_m}   
\exp\pr{-\pr{\frac{x}{y}}^{\frac{2}{m+1}}}
\\
+b_m \int_{S^{N-m}}  \int_1^\infty  \int_1^\infty 
 uv\pr{1+\log v}^{p_m}    \\
\PP\pr{ 
\sum_{\gri\in\inc^m_{N-1}}\sum_{j=i_m}^N
 \norm{ h_{\gri,j}
 \pr{\xi_{\gri}^{\dec},s_{j } } }_{ \Hi}^2>\frac{yu vc_m}{8\pr{1+\log u}^{m/2}} 
    } dudv
d\PP_{\xi_{\intent{m+1,N}}^{\pr{m+1}}}\pr{s_{\intent{m+1,N}}}
\\
=  \pr{4+2a_m}   
\exp\pr{-\pr{\frac{x}{y}}^{\frac{2}{m+1}}}
+b_m \int_1^\infty  \int_1^\infty  uv\pr{1+\log v}^{p_m}g\pr{\frac{yu 
vc_m}{\pr{1+\log u}^{m/2}} }
dudv,
\end{multline} 
where 
\begin{equation*}
 g\pr{t}= \PP\pr{ \sqrt{ \sum_{\gri\in\inc^{m+1}_{N}}
 \norm{ h_{\gri}
 \pr{\xi_{\gri}^{\dec} }  
}_{ \Hi}^2} >\frac t8  }.
\end{equation*}
Doing the substitution $w=u\pr{1+\log u}^{-m/2}v$ for a fixed $u\geq 1$, it 
suffices to prove the existence of constants $\kappa_{m}$ and  
$\kappa'_{m}$ for which for each $w\in\R$, 
\begin{multline}\label{eq:last_integral_bound}
I\pr{w}:= \int_1^\infty \ind{u\pr{1+\log u}^{-m/2}<w}\frac 1u\pr{1+\log u}^m
 \pr{1+\log\pr{\frac{w\pr{1+\log u}^{m/2}}{u}}  }^{p_m}du
\\
 \leq \kappa_m \ind{w>\kappa'_m}\pr{1+\log w}^{p_{m+1}}.
\end{multline}
Let us show \eqref{eq:last_integral_bound}. From the observation that there 
exists a constant $\kappa_{1,m}$ 
depending only on $m$ such that for each $u\geq 1$, $\sqrt{u}\leq \kappa_{1,m} 
u\pr{1+\log u}^{-m/2}$, we derive that $I\pr{\omega}$ vanishes if 
$w<\pr{\kappa_{1,m}}^{-1}$, hence 
\begin{equation*} 
 I\pr{w}
 \leq   \ind{w>\pr{\kappa_{1,m}}^{-1}} 
 \int_1^{w^2\pr{\kappa_{1,m}}^2}\frac 1u\pr{1+\log u}^m
 \pr{1+\log\pr{\frac{w\pr{1+\log u}^{\frac m2}}{u}}  }^{p_m}du.
\end{equation*}
Moreover, using that $\pr{1+\log u}^{m/2}/u$ can be bounded independently of 
$u\geq 1$, we get for some constant $\kappa_{2,m}$ that 
\begin{multline}\label{eq:last_integral_bound_ter}
I\pr{w}
 \leq   \kappa_{2,m}\ind{w>1} 
 \int_1^{w^2\pr{\kappa_{1,m}}^2}\frac 1u\pr{1+\log u}^m
 \pr{1+ \log w  }^{p_m}du\\+\kappa_{2,m}
 \ind{\pr{\kappa_{1,m}}^{-1}<w\leq 1}
 \int_1^{ \pr{\kappa_{1,m}}^2}\frac 1u\pr{1+\log u}^m
 \pr{1+\log\pr{\frac{\pr{1+\log u}^{\frac m2}}{u}}  }^{p_m}du
\end{multline}
 which gives \eqref{eq:last_integral_bound}. 
 This ends the proof of Theorem~\ref{thm:ineg_exp_deg_Ustat}.
 \subsection{Proof of Theorem~\ref{thm:dev_ineg_fixed_kernel}}

Starting from \eqref{eq:decomposition_Hoeffding_deg_d}, one has 
\begin{equation*}
 \max_{m\leq n\leq N}\norm{U_{m,n}\pr{h}}_{\Hi}\leq 
\sum_{k=d}^{m}\binom{m}{k} \max_{k\leq n\leq N}
\norm{U_{k,n}\pr{h_k}}_{\Hi}
\frac{\binom nm}{\binom nk} \leq \kappa_m \sum_{k=d}^{m}N^{m-k}\max_{k\leq 
n\leq 
N}\norm{U_{k,n}\pr{h_k}}_{\Hi},
\end{equation*}
 where $h_k$ is defined as in 
\eqref{eq:def_hk} and $\kappa_m$ depends only on $m$, hence 
\begin{equation}\label{eq:tail_after_Hoeffding}
\PP\pr{\max_{m\leq n\leq N}\norm{U_{m,n}\pr{h}}_{\Hi}>xN^{m-\frac d2}}
\leq  \sum_{k=d}^{m}
\PP\pr{\max_{k\leq n\leq N}\norm{U_{k,n}\pr{h_k}}_{\Hi}>xN^{k-\frac 
d2}\kappa_m/m   
}.
\end{equation}
For each $k\in\intent{d,m}$, we apply Theorem~\ref{thm:ineg_exp_deg_Ustat} in 
the following 
context: $m_k=k$, $h_{\gri}=h_k$, 
$x_k=xN^{k-d/2}\kappa_m/m$ and 
$y_k=N^{k-d/2}x^{1-k/d}\kappa_m/m$. We obtain 
in view of \eqref{eq:tail_after_Hoeffding} that 
\begin{multline*}
\PP\pr{\max_{m\leq n\leq N}\norm{U_{m,n}\pr{h}}_{\Hi}>xN^{m-\frac d2}}
\leq  \sum_{k=d }^{m} A_k\exp\pr{-\pr{\frac{x}{y}}^{\frac 2 d }}\\+B_k
\int_1^\infty u\pr{1+ \log u  
}^{\frac{m\pr{m+1}}2-1}\PP\pr{\sum_{\gri\in\inc^k_N     }
\norm{h_k\pr{\xi_{\gri}^{\dec}  }   }_{\Hi}^2> 
\pr{\frac{C_k}mN^{k-\frac d2}x^{1-\frac kd}\kappa_mu}^2 }du,
\end{multline*}
where $A_k,B_k,C_k,k\in\intent{d+1,m}$ are as in 
Theorem~\ref{thm:ineg_exp_deg_Ustat}.
 Also, since the random variables 
$\norm{h_k\pr{\xi_{\gri}^{\dec}  }   }_{\Hi}^2$, 
$\gri\in\inc^k_N  $, have the 
 same distribution, one has 
 \begin{equation}
 \sum_{\gri\in\inc^k_N     }
\norm{h_k\pr{\xi_{\gri}^{\dec} }   }_{\Hi}^2\conv 
N^k\norm{h_k\pr{\xi_{\intent{1,k}   } }}_{\Hi}^2\conv N^k\kappa_k H_k, 
 \end{equation}
where $\conv$ is defined as in Definition~\ref{def:conv_order} and 
$\kappa_k$ depends only on $k$.
Therefore, we infer from \eqref{eq:conv_ordering_tails_integrales} that 
\begin{multline*}
\PP\pr{\max_{m\leq n\leq N}\norm{U_{m,n}\pr{h}}_{\Hi}>xN^{m-\frac d2}}
\leq  \sum_{k=d}^{m} A_k\exp\pr{-\pr{\frac{x}{y}}^{\frac 
2d}}\\+\sum_{k=d}^{m}B_k
\int_1^\infty\int_1^\infty u\pr{1+ \log u  
}^{p_m}\PP\pr{ N^k
\norm{h_k\pr{\xi_{\intent{1,k}}   }   }_{\Hi}^2> 
\pr{\frac{C_k}mN^{\frac{2k-d}2}x^{\frac{d-k}d}\kappa_mu}^2\frac{v}{4} }dudv\\
=\sum_{k=d}^{m} A_k\exp\pr{-\pr{\frac{x}{y}}^{\frac 2d}} +\sum_{k=d}^{m}B_k
\int_1^\infty\int_1^\infty u\pr{1+ \log u  
}^{p_m}\PP\pr{  
H > 
 \frac{C_k\kappa_m}{2m}N^{\frac{k-d}2}x^{\frac{d-k}d} u \sqrt{v} }dudv.
\end{multline*}
Lemma~2.2 in \cite{MR4294337} gives 
\begin{equation}\label{eq:borne_integrale_double}
\int_1^\infty\int_1^\infty u\pr{1+ \log u  }^{p_m}
g\pr{u\sqrt v}dudv
\leq \kappa_m \int_1^\infty u \pr{1+ \log u  }^{p_{m=1}}
g\pr{u/\kappa_m},
\end{equation}
which concludes the proof of Theorem~\ref{thm:dev_ineg_fixed_kernel}.
 
\subsection{Proof of Corollary~\ref{cor:weighted_Ustat_non_deg}}
Starting from the decomposition \eqref{eq:Hoeffding_weighted}, 
one has 
\begin{equation*}
\PP\pr{\norm{\sum_{\gri\in\inc_n^m 
}T_{\gri}\pr{h\pr{\xi_{\gri} }} }_{\Hi}>x } \leq 
\sum_{k=d}^m\PP\pr{\norm{\sum_{\gri\in\inc_k^m 
}a_{\gri}^{\pr{n,k}}\pr{h_k\pr{\xi_{\gri }}} }_{\Hi}>x/m 
}.
\end{equation*}
For each fixed $k\in\intent{d,m}$, we apply 
Theorem~\ref{thm:ineg_exp_deg_Ustat} with $m$ replaced by $k$, 
$y_k=x^{1-k/d}y^{k/d}$ and 
$h_{\gri}\pr{s_{1},\dots,s_k}=a_{\gri}^{\pr{n,k}}\pr{h_k\pr{
s_1,\dots,s_k}}$.
We derive that 
\begin{multline*}
\PP\pr{\norm{\sum_{\gri\in\inc_n^m 
}T_{\gri}\pr{h\pr{\xi_{\gri}}} }_{\Hi}>x } \leq 
\sum_{k=d}^mA_k
\exp\pr{-\pr{\frac{x}{y}}^{\frac 2d }}\\+
\sum_{k=d}^mB_k\int_1^{\infty}u\pr{1+\log u}^{k\pr{k-1}/2-1} 
\PP\pr{\sqrt{ 
\sum_{\gri\in\inc_k^m }\norm{a_{\gri}^{\pr{n,k}}
\pr{h_k\pr{\xi_{\gri}^{\dec} }}}_{\Hi}^2} 
>\frac{C_k}mx^{\frac{d-k}{d} }y^{\frac k d }u  }du.
\end{multline*}
For a convex increasing function $\varphi\colon\R\to\R$, one has 
\begin{align*}
\E{\varphi\pr{ \sum_{\gri\in\inc_k^m 
}\norm{a_{\gri}^{\pr{n,k}}\pr{h_k\pr{\xi_{\gri}^{\dec} }}  }_{\Hi}^2}}
&\leq \E{\varphi\pr{A_{n,k}^2 \sum_{\gri\in\inc_k^m 
}\frac{\norm{a_{\gri}^{\pr{n,k}}}_{\Bca\pr{\Hi}}^2}{A_{n,k}^2} 
\norm{h_k\pr{\xi_{\gri}^{\dec} }  }  _{\Hi}^2}}\\
&\leq \E{ \sum_{\gri\in\inc_k^m 
}\frac{\norm{a_{\gri}^{\pr{n,k}}}_{\Bca\pr{\Hi}}^2}{A_{n,k}^2} 
\varphi\pr{A_{n,k}^2\norm{h_k\pr{\xi_{\gri}^{\dec} } 
}  
_{\Hi}^2}}\\
&=\E{\varphi\pr{A_{n,k}^2\norm{h_k\pr{\xi_{\intent{1,k} }}}_{\Hi}^2}}\\
&\leq \E{\varphi\pr{A_{n,k}^2\kappa_m H^2}}\\
&\leq \E{\varphi\pr{ 
n^{m-k}\sum_{\gri\in\inc^m_n}\norm{T_{\gri}}_{\Bca\pr{\Hi}}^2\kappa_m H^2}}\\
&\leq \E{\varphi\pr{ 
n^{m-d}\sum_{\gri\in\inc^m_n}\norm{T_{\gri}}_{\Bca\pr{\Hi}}^2\kappa_m H^2}}
\end{align*}
hence a use of \eqref{eq:conv_ordering_tails_integrales} gives 
\begin{multline*}
\PP\pr{\norm{\sum_{\gri\in\inc_n^m 
}T_{\gri}\pr{h\pr{\xi_{\gri} }} }_{\Hi}>x } \leq 
\sum_{k=d}^mA_k
\exp\pr{-\pr{\frac{x}{y}}^{2/ d}}\\+
\sum_{k=d}^mB_k\int_1^{\infty}\int_1^{\infty}u\pr{1+\log u}^{k\pr{k-1}/2-1}
\PP\pr{A_{n,k}\norm{h_k\pr{\xi_{\intent{1,k}}}}_{\Hi}
>x^{\frac{d-k}d} y^{\frac kd}u\sqrt{v}/\pr{4m} }dudv.
\end{multline*}
We conclude by a use of \eqref{eq:borne_integrale_double}. 

\subsection{Proof of the results of Subsection~\ref{subsec:incomplete}}

Before going into the proofs of each corollary, let us mention the common 
elements of proof.

Define for $k\in\intent{d,m}$ and $\gri\in\inc^k_n$ the random variable
\begin{equation*}
Z_{\gri}^{\pr{k}}:=\sum_{\grj \in  I_k\pr{\gri}}Z_{n,\gr{j}},
\end{equation*}
where $ I_k\pr{\gri}$ denotes the set of indices 
$\grj=\pr{j_\ell}_{\ell=1}^m\in\inc^m_n$ such that on a set 
$K=\ens{\ell_1,\dots,\ell_k}$ of cardinal $k$, 
one has $j_{\ell_1}=i_1,\dots,j_{\ell_k}=i_k$. By symmetry of $h$, we have 
\begin{equation}
\sum_{\gri\in\inc_n^m}Z_{n,\gri}h\pr{\xi_{\gri}}
=\sum_{k=d}^m\sum_{\gri\in\inc_n^k}Z_{n,\gri}^{\pr{k}}h_k\pr{\xi_{\gri}}.
\end{equation}
This decomposition reduces us to show the existence of constants $a_m,b_m$ and 
$c_m$ depending only on $m$ such that for each $k\in\intent{d,m}$ and $x,y$ 
such that $x/y>1$, 
\begin{multline}\label{eq:ineg_pour_cor_inc_replacement}
 \PP\pr{\norm{\sum_{\gri\in\inc_n^k}Z_{n,\gri}^{\pr{k}}h_k\pr{\xi_{\gri}}    
}_{\Hi}>x \sqrt{N}\sqrt{\min\ens{N,n^{m-d}  }}  }\leq 
a_m\exp\pr{-\pr{\frac{x}{y}}^{\frac 2m}}\\
+b_m \int_1^\infty  u\pr{1+\log u}^{m\pr{m+1}/2}
\PP\pr{ H>c_m yu}du\mbox{ and }
\end{multline}
\begin{multline} \label{eq:ineg_pour_cor_inc_Bernoulli}
 \PP\pr{\norm{\sum_{\gri\in\inc_n^k}Z_{n,\gri}^{\pr{k}}h_k\pr{\xi_{\gri}}    
}_{\Hi}>x n^{m}\sqrt{p_n}\sqrt{\min\ens{p_n,n^{-d}  }}  }
\\
\leq 
a_m\exp\pr{-\frac{n^mp_n^2}2}+ a_m\exp\pr{-\pr{\frac{x}{y}}^{\frac 2m}}
+b_m \int_1^\infty  u\pr{1+\log u}^{m\pr{m+1}/2}
\PP\pr{H>c_m yu}du.
\end{multline}
\begin{proof}[Proof of 
Corollary~\ref{cor:deviation_incomplete_Ustats_replacement}]
By independence between $\pr{Z_{n,\gri}}_{\gri\in\inc^m_n}$ and 
$\pr{\xi_i}_{i\geq 1}$, 
one has 
\begin{multline*}
\PP\pr{\norm{\sum_{\gri\in\inc_n^k}Z_{n,\gri}^{\pr{k}}h_k\pr{\xi_{\gri}}}_{\Hi}>
x}\\
=\sum_{\pr{z_{\gri}}_{\gri\in\inc^k_n}\in \ens{0,1}^{C^k_n}  }
\PP\pr{\norm{\sum_{\gri\in\inc_n^k}  \sum_{\grj \in  I_k\pr{\gri}} z_{\grj}  
h_k\pr{\xi_{\gri}}}_{\Hi}>x}\PP\pr{\pr{Z_{n,\gri}}_{\gri\in\inc_n^k} 
=\pr{z_{\gri}}_{\gri\in\inc_n^k}}.
\end{multline*}
Using Theorem~\ref{thm:ineg_exp_deg_Ustat} for fixed 
$\pr{z_{\gri}}_{\gri\in\inc_n^k}$, 
we find that for $y_k$ that will be specified later, 
\begin{multline*}
\PP\pr{\norm{\sum_{\gri\in\inc_n^k}Z_{n,\gri}^{\pr{k}}h_k\pr{\xi_{\gri}}}_{\Hi}>
x}
\leq a_k\exp\pr{-\pr{\frac x{y_k}}^{\frac{2}{k}}}\\
+b_k\sum_{\pr{z_{\gri}}_{\gri\in\inc^k_n}\in \ens{0,1}^{C^k_n}  
}\PP\pr{\pr{Z_{n,\gri}}_{\gri\in\inc_n^k} =\pr{z_{\gri}}_{\gri\in\inc_n^k}}\\  
\int_1^\infty u \pr{1+\log u}^{k\pr{k-1}/2}
\PP\pr{\sum_{\gri\in\inc_n^k}  \norm{\sum_{\grj \in  I_k\pr{\gri}} z_{\grj}  
h_k\pr{\xi^{\dec}_{\gri}}}_{\Hi}^2>u^2y_k^2}du.
\end{multline*}
Let $A:=\sum_{\gri\in\inc_n^k}\pr{\sum_{\grj \in  I_k\pr{\gri}} z_{\grj} }^2  $.
For a convex increasing function $\varphi\colon\R\to\R$, one has 
\begin{align*}
\E{\varphi\pr{{\sum_{\gri\in\inc_n^k}  \norm{\sum_{\grj \in  I_k\pr{\gri}} 
z_{\grj}  h_k\pr{\xi_{\gri}^{\dec}}}_{\Hi}^2}}}&=
\E{\varphi\pr{\frac 1A\sum_{\gri\in\inc_n^k}\pr{\sum_{\grj \in  I_k\pr{\gri}} 
z_{\grj} }^2  A\norm{ h_k\pr{\xi^{\dec}_{\gri}}}_{\Hi}^2}}\\
&\leq \E{\frac 1A\sum_{\gri\in\inc_n^k}\pr{\sum_{\grj \in  I_k\pr{\gri}} 
z_{\grj} }^2  \varphi\pr{A\norm{ h_k\pr{\xi^{\dec}_{\gri}}}_{\Hi}^2}}\\
&\leq \E{\frac 1A\sum_{\gri\in\inc_n^k}\pr{\sum_{\grj \in  I_k\pr{\gri}} 
z_{\grj} }^2  \varphi\pr{A \kappa_mH^2}}
\end{align*}
hence $\sum_{\gri\in\inc_n^k}  \norm{\sum_{\grj \in  I_k\pr{\gri}} z_{\grj}  
h_k\pr{\xi_{\gri}}}_{\Hi}^2 \conv A \kappa_mH^2$. Since $\sqrt{\sum c_i} \leq 
\sum \sqrt{c_i}$, one has $A\leq N^2$. Moreover, by Cauchy-Schwarz, $A\leq 
n^{m-k}N\leq 
n^{m-d}N$. Overall, 
\begin{equation*}
\sum_{\gri\in\inc_n^k}  \norm{\sum_{\grj \in  I_k\pr{\gri}} z_{\grj}  
h_k\pr{\xi^{\dec}_{\gri}}}_{\Hi}^2 \conv N\min\ens{N,n^{m-d}} \kappa_mH^2
\end{equation*}
hence by \eqref{eq:conv_ordering_tails_integrales} and 
\eqref{eq:borne_integrale_double}, we get 
\begin{multline*}
\PP\pr{\norm{\sum_{\gri\in\inc_n^k}Z_{n,\gri}^{\pr{k}}h_k\pr{\xi_{\gri}}}_{\Hi}>
x}
\leq a_k\exp\pr{-\pr{\frac x{y_k}}^{2/k}}\\
+b_k \int_1^\infty u \pr{1+\log u}^{k\pr{k+1}/2} \PP\pr{  \sqrt{ 
N\min\ens{N,n^{m-d}}}\kappa'_mH>y_ku }du.
\end{multline*}
Bounding the exponent $k\pr{k+1}/2$ by $m\pr{m+1}/2$, choosing $y_k$ such that 
$\pr{\frac x{y_k}}^{2/k}=\pr{\frac x{y}}^{2/m}$ and using the fact that 
$x/y>1$, one gets 
\begin{multline*}
\PP\pr{\norm{\sum_{\gri\in\inc_n^k}Z_{n,\gri}^{\pr{k}}h_k\pr{\xi_{\gri}}}_{\Hi}>
x}
\leq a_k\exp\pr{-\pr{\frac x{y}}^{2/m}}\\
+b_k \int_1^\infty u \pr{1+\log u}^{m\pr{m+1}/2} \PP\pr{  \sqrt{ 
N\min\ens{N,n^{m-d}}}\kappa'_mH>y u }du
\end{multline*}
from which we derive \eqref{eq:ineg_pour_cor_inc_replacement} by applying the 
previous 
inequality with $x$ and $y$ replaced respectively  by $x 
\pr{N\min\ens{N,n^{m-d}}}^{-1/2}$  
and $y \pr{N\min\ens{N,n^{m-d}}}^{-1/2}$.
\end{proof}
\begin{proof}[Proof of 
Corollary~\ref{cor:deviation_incomplete_Ustats_Bernoulli}]
We start from 
\begin{equation*}
\PP\pr{\norm{\sum_{\gri\in\inc_n^k}Z_{n,\gri}^{\pr{k}}h_k\pr{\xi_{\gri}}    
}_{\Hi}>x n^{m}\sqrt{p_n}\sqrt{\min\ens{p_n,n^{-d}  }}  }\leq P_1+P_2,
\end{equation*}
where 
\begin{equation*}
P_1:=\PP\pr{\sum_{\gri\in\inc^m_n}Z_{n,\gri}>2n^mp_n  } \mbox{ and}
\end{equation*}
\begin{equation*}
P_2:=\PP\pr{\ens{\norm{\sum_{\gri\in\inc_n^k}Z_{n,\gri}^{\pr{k}}h_k\pr{\xi_{\gri
}}    
}_{\Hi}>x n^{m}\sqrt{p_n}\sqrt{\min\ens{p_n,n^{-d}  }}  }
\cap \ens{\sum_{\gri\in\inc^m_n}Z_{n,\gri}\leq 2n^mp_n  } }.
\end{equation*}
The term $P_1$ is bounded,
via Azuma-Hoeffding's inequality, by $2\exp\pr{-n^mp_n^2/2}$. 
Using independence between $\pr{Z_{n,\gri}}_{\gri\in\inc^m_n}$ and 
$\pr{\xi_i}_{i\geq 1}$, we get 
\begin{multline*}
P_2=\sum_{\pr{z_{\gri}}_{\gri\in\inc^k_n}\in \ens{0,1}^{C^k_n}  }
\PP\pr{\norm{\sum_{\gri\in\inc_n^k}  \sum_{\grj \in  I_k\pr{\gri}} z_{\grj}  
h_k\pr{\xi_{\gri}}}_{\Hi}>xn^{m}\sqrt{p_n}\sqrt{\min\ens{p_n,n^{-d}  }}}
\\ \ind{
\sum_{\gri\in\inc^m_n}z_{\gri}\leq 2n^mp_n } 
\PP\pr{\pr{Z_{n,\gri}}_{\gri\in\inc_n^k} =\pr{z_{\gri}}_{\gri\in\inc_n^k}}.
\end{multline*}
We apply Theorem~\ref{thm:ineg_exp_deg_Ustat} for fixed 
$\pr{z_{\gri}}_{\gri\in\inc^k_n}$ and get 
\begin{multline*}
P_2\leq a_k\exp\pr{-\pr{\frac{x}{y_k}}^{2/k}}
+b_k \int_1^{\infty} u\pr{1+\log 
u}^{\frac{k\pr{k-1}}{2}}\sum_{\pr{z_{\gri}}_{\gri\in\inc^k_n}\in 
\ens{0,1}^{C^k_n}  }\PP\pr{\pr{Z_{n,\gri}}_{\gri\in\inc_n^k} 
=\pr{z_{\gri}}_{\gri\in\inc_n^k}} \\ \PP\pr{\sum_{\gri\in\inc_n^k}  
\norm{\sum_{\grj \in  I_k\pr{\gri}} z_{\grj}  
h_k\pr{\xi^{\dec}_{\gri}}}_{\Hi}^2>u^2\pr{xn^{m}}^2 p_n \min\ens{p_n,n^{-d}  }} 
\ind{
\sum_{\gri\in\inc^m_n}z_{\gri}\leq 2n^mp_n } 
du.
\end{multline*}
For $\pr{z_{\gri}}_{\gri\in\inc^m_n}$ such that 
$\sum_{\gri\in\inc^m_n}z_{\gri}\leq 2n^mp_n$ and a convex non-decreasing 
function $\varphi\colon\R\to\R$, letting 
$A:=\sum_{\gri\in\inc_n^k} \pr{\sum_{\grj \in  I_k\pr{\gri}}}^2$, one has 
\begin{align*}
\E{\varphi\pr{\sum_{\gri\in\inc_n^k}  \norm{\sum_{\grj \in  I_k\pr{\gri}} 
z_{\grj}  h_k\pr{\xi^{\dec}_{\gri}}}_{\Hi}^2}}&=
\E{\varphi\pr{\frac 1A\sum_{\gri\in\inc_n^k} \pr{\sum_{\grj \in  I_k\pr{\gri}} 
z_{\grj}}^2  A\norm{ h_k\pr{\xi^{\dec}_{\gri}}}_{\Hi}^2}}\\
&\leq 
\E{\frac 1A\sum_{\gri\in\inc_n^k} \pr{\sum_{\grj \in  I_k\pr{\gri}} z_{\grj}}^2 
\varphi\pr{ A\norm{ h_k\pr{\xi^{\dec}_{\gri}}}_{\Hi}^2}}\\
&\leq \E{\varphi\pr{ A\kappa_mH^2}}.
\end{align*}
Since $\sum_{\gri\in\inc^m_n}z_{\gri}\leq 2np_n$, an application of 
 $\sqrt{\sum c_i} \leq \sum \sqrt{c_i}$ gives that $A^2\leq 2n^mp_n$. Moreover, 
 convexity of the square gives $A\leq n^{m-k} 2n^mp_n\leq 2n^{2m-d}p_n$. The 
combination of the obtained estimates gives 
 $\sum_{\gri\in\inc_n^k}  \norm{\sum_{\grj \in  I_k\pr{\gri}} z_{\grj}  
h_k\pr{\xi^{\dec}_{\gri}}}_{\Hi}^2\conv \kappa_m
  \min\ens{\sqrt{2n^mp_n},2n^{2m-d}p_n} H^2$. We then conclude similarly as in 
the proof of Corollary~\ref{cor:deviation_incomplete_Ustats_replacement}: we 
use 
  \eqref{eq:conv_ordering_tails_integrales}, then 
\eqref{eq:borne_integrale_double}
  and take $y_k$ making the terms $\exp\pr{-\pr{\frac{x}{y_k}}^{2/k}}, 
k\in\intent{d,m}$, equal and finally, the fact that $x/y$ is bigger than one 
allows to find a bound independent 
  of $k$. This ends the proof of 
Corollary~\ref{cor:deviation_incomplete_Ustats_Bernoulli}.
\end{proof}
\begin{appendix}
 \section{A deviation inequality for Hilbert-valued martingale differences}
 \subsection{Real-valued case}
 In this  section, we collect existing results in the literature in order to 
derive 
an exponential inequality for real-valued martingale difference sequence. The 
tail of the maxima of absolute values of a martingale are controlled via an 
exponential term and the tail of the sum of squares. This will be the starting 
point

\cite{MR3311214} showed the following.
\begin{Theorem}[Theorem~2.1 in \cite{MR3311214}]\label{thm:Fan_Grama_Liu}
Let $\pr{D_j}_{j\geq 1}$ be a real-valued martingale difference sequence 
with respect to the filtration $\pr{\Fca_j}_{j\geq 0}$. Suppose that 
there exists random variables $V_{j-1}$, $j\in\intent{1,n}$, which are 
non-negative and
$\Fca_{j-1}$-measurable, non-negative functions $f$ and $g$ such that for some 
positive $\lambda$ and for all $j\in\intent{1,n}$,
\begin{equation}\label{eq:condition_sur_les_accroissements}
\E{\exp\pr{\lambda D_j-g\pr{\lambda}D_j^2}\mid \Fca_{j-1}}\leq 1+f\pr{\lambda}
V_{j-1}.
\end{equation}
Then for all $x$, $v$, $w>0$, 
\begin{equation}\label{eq:inegalite_Fan_grama_Liu}
\PP\pr{\bigcup_{k=1}^n \ens{\sum_{j=1}^kD_j\geq x}\cap 
\ens{\sum_{j=1}^kD_j^2\leq v^2   }\cap\ens{\sum_{j=1}^kV_{j-1}\leq w   }}\\
\leq \exp\pr{-\lambda x+g\pr{\lambda}v^2+f\pr{\lambda}w}.
\end{equation}
\end{Theorem}
 
 Notice that the condition \eqref{eq:condition_sur_les_accroissements} is 
satisfied 
 for all $\lambda>0$ when $f\pr{\lambda}=g\pr{\lambda}=\lambda^2/2$ and 
 $V_{j-1}=\E{D_j^2\mid\Fca_{j-1}}$ provided that $\E{D_j^2}$ is finite for each 
$j$. After having chosen $w=v^2$, optimized over 
 $\lambda$ the right hand side of \eqref{eq:inegalite_Fan_grama_Liu} 
 and applying Theorem~\ref{thm:Fan_Grama_Liu} to $D_j$ and then to 
  $-D_j$, we derive that  for all $x$ and $y>0$
 \begin{equation}\label{eq:ineg_exp_martingales_valeurs_reelle}
 \PP\pr{\max_{1\leq k\leq n}\abs{\sum_{j=1}^kD_j}>x   }
 \leq 2\exp\pr{-\frac{x^2}{ y^2}}+\PP\pr{\sum_{j=1}^n
 \pr{D_j^2+\E{D_j^2\mid \Fca_{j-1}}}>\frac{y^2}2}.
 \end{equation}
 An inequality in this spirit has been established by \cite{MR2462551}, 
 Theorem~2.1, without the max but the event involving the sum of squares 
 and conditional variances is in the probability in the left hand side.
 
 \subsection{Dimension reduction in Hilbert spaces}

An extension of \eqref{eq:inegalite_Fan_grama_Liu} to 
the Hilbert-valued case requires the following dimension reduction result.
This was done by \cite{MR1096481} in the continuous time case, 
and in Proposition 5.8.3 in \cite{MR1167198} in the discrete case.

\begin{Lemma}\label{lem:dimension_reduction}
Let $\pr{\Hi,\langle\cdot,\cdot\rangle}$ be a separable Hilbert space, let 
$\norm{x}_{\Hi}:=\sqrt{\langle x,x\rangle}$ and let $\pr{D_j}_{j\geq 1}$ be 
a $\Hi$-valued martingale difference sequence with respect to the filtration 
$\pr{\Fca_j}_{j\geq 0}$. Enlarging the probability space if needed, there 
exists an $\R^2$-valued
martingale difference sequence $\pr{\pr{d_{1,j},d_{2,j}}}_{j\geq 1}$ with 
respect to the filtration $\pr{\Gca_j}_{j\geq 0}$
such that for each $j\geq 1$, $\norm{D_j}_{\Hi}^2=
 d_{1,j}^2+d_{2,j}^2$ and
for each $n\geq 1$, 
\begin{equation}\label{eq:dim_red_cond_var}
\sum_{j=1}^n\E{\norm{D_j}_{\Hi}^2\mid\Fca_{j-1}}=\sum_{i=1}^n
\E{d_{j,1}^2\mid\Gca_{j-1}}
 +\sum_{j=1}^n\E{d_{j,2}^2\mid\Gca_{j-1}} \mbox{ and}
\end{equation}
\begin{equation}\label{eq:dim_red_norm_martingale}
\norm{\sum_{i=1}^nD_j}_{\Hi}^2=  \pr{\sum_{j=1}^nd_{1,j} }^2+
\pr{\sum_{j=1}^nd_{2,j} }^2  .
\end{equation}
\end{Lemma}
A consequence of such a dimension reduction is that 
a deviation inequality for real valued martingale difference sequences 
in terms of tails of norms of increments can be covnerted into 
a similar inequality for  
a Hilbert-valued martingale difference sequences. 
An inequality in the spirit of \eqref{eq:ineg_exp_martingales_valeurs_reelle} 
reads as follows.

\begin{Proposition}\label{prop:ineg_deviation_sans_var_cond_cas_Hilbert}
For each separable Hilbert space  $\pr{\Hi,\langle\cdot,\cdot\rangle}$ and 
each $\Hi$-valued martingale difference sequence 
$\pr{D_j}_{j\geq 1}$ with respect to a filtration $\pr{\Fca_{j}}_{j\geq 0}$, 
 \begin{equation}\label{eq:ineg_deviation_martingale_cas_Hilbert}
 \PP\pr{\max_{1\leq k\leq n}\norm{\sum_{j=1}^kD_j}_{\Hi}>x   }
  \leq 4\exp\pr{-\frac{x^2}{y^2}}+
 2 \PP\pr{ \sum_{j=1}^n
  \pr{\norm{D_j}_{\Hi}^2+\E{\norm{D_j}_{\Hi}^2\mid\Fca_{j-1}}  }
>\frac{y^2}{8} } ,
\end{equation}
where 
$\norm{x}_{\Hi}=\sqrt{\langle x,x\rangle}$. 
\end{Proposition}

\begin{proof}
Let $\Hi$ be a separable Hilbert space and let $\pr{D_j}_{j\geq 1}$ 
be a $\Hi$-valued martingale difference sequence with respect to a filtration 
$\pr{\Fca_{j}}_{j\geq 0}$. By
Lemma~\ref{lem:dimension_reduction}, there exist an $\R^2$-valued martingale
difference sequence $\pr{\pr{d_{1,j},d_{2,j}}}_{j\geq 1}$ 
such that for each $j\geq 1$, $\norm{D_j}_{\Hi}^2=
 d_{1,j}^2+d_{2,j}^2 $ and \eqref{eq:dim_red_cond_var}, 
\eqref{eq:dim_red_norm_martingale} hold  for 
each $n$. In particular, 
 \begin{equation}\label{eq:dim_red_bound_norm_martingale}
\norm{\sum_{j=1}^nD_j}_{\Hi}\leq \abs{\sum_{j=1}^nd_{1,j} }+
\abs{\sum_{j=1}^nd_{2,j} }
\end{equation}
from which the following inequality follows:
\begin{equation*}
 \PP\pr{\max_{1\leq k\leq n}\norm{\sum_{j=1}^kD_j}_{\Hi}>x   }
 \leq  \PP\pr{\max_{1\leq k\leq n}\abs{\sum_{j=1}^kd_{1,j}}>\frac{x}{2}   }
 +\PP\pr{\max_{1\leq k\leq n}\abs{\sum_{j=1}^kd_{2,j}}>\frac{x}{2}   }.
\end{equation*}
An application of  \eqref{eq:ineg_exp_martingales_valeurs_reelle}
with 
$y^2$ replaced by $y_1^2=y^2/4$ gives
\begin{multline*}
\PP\pr{\max_{1\leq k\leq n}\norm{\sum_{j=1}^kD_j}_{\Hi}>x   }
 \leq 4\exp\pr{-\frac{x^2}{y^2}}\\+
 \PP\pr{\sum_{j=1}^n d_{1,j}^2+\E{d_{1,j}^2\mid\Gca_{j-1}}>\frac{y^2}{8}  }
 + \PP\pr{\sum_{j=1}^n d_{2,j}^2+\E{d_{2,j}^2\mid\Gca_{j-1}}>\frac{y^2}{8}  }
\end{multline*} 
and \eqref{eq:ineg_deviation_martingale_cas_Hilbert} follows from the 
combination of \eqref{eq:dim_red_cond_var} and 
\eqref{eq:dim_red_norm_martingale}.
\end{proof}
The following ordering was studied in \cite{MR606989}.
\begin{Definition}\label{def:conv_order}
 Let $X$ and $Y$ be two real-valued random variables. We say that $X\conv Y$ if
for each nondecreasing convex $\varphi\colon\R\to\R$ such that the expectations
$\E{\varphi\pr{X}}$ and $\E{\varphi\pr{Y}}$ exist, $\E{\varphi\pr{X}}
\leq \E{\varphi\pr{Y}}$.
\end{Definition}

The point of this ordering is that if $X\conv Y$, one can formulate a tail
inequality for $X$ in terms of tails of $Y$. More precisely, Lemma~2.1 in
\cite{MR4294337} gives that if $X$ and $Y$ are nonnegative random variables
such that $X\conv Y$, then for each positive  $t$,
\begin{equation}\label{eq:conv_ordering_tails_integrales}
 \PP\pr{X>t}\leq \int_1^\infty\PP\pr{Y>t\frac{v}{4}}dv.
\end{equation}

When the conditional variances admit a particular form, it
is  possible to express the tails of maxima of a martingale
by the sum of squares of the norm of the increments, that is, without 
conditional moment.
\begin{Proposition}
\label{prop:ineg_deviation_martingale_cas_Hilbert_hyp_var_cond}
 Let $\pr{\Hi,\scal{\cdot}{\cdot}_{\Hi}}$ be a separable Hilbert space. Let 
$\pr{D_j}_{j\geq 1}$ be a square integrable martingale difference sequence with 
respect to a filtration $\pr{\Fca_{j}}_{j\geq 0}$. Suppose that
 \begin{equation*}
  \E{\norm{D_j}_{\Hi}^2\mid \Fca_{j-1}}
 =\E{\norm{D_j}_{\Hi}^2\mid \Fca_{0}}\mbox{ a.s.}.
 \end{equation*}
Then for each positive $x$ and $y$, the following inequality
holds:
 \begin{equation}\label{eq:ineg_deviation_martingale_cas_Hilbert_hyp_var_cond}
 \PP\pr{\max_{1\leq k\leq n}\norm{\sum_{i=1}^kD_j}_{\Hi}>x   }
 \leq 4\exp\pr{-\frac{x^2}{y^2}}+
 4\int_1^\infty u\PP\pr{ \sqrt{\sum_{j=1}^n
  \norm{D_j}_{\Hi}^2} 
>y \frac{u}{8} }du.
\end{equation}
\end{Proposition}
\begin{proof}
We combine \eqref{eq:ineg_deviation_martingale_cas_Hilbert} 
with \eqref{eq:conv_ordering_tails_integrales}, applied with 
$X=\sum_{j=1}^n\pr{\norm{D_j}_{\Hi}^2+\E{\norm{D_j}_{\Hi}^2\mid\Fca_{j-1}}  }$ 
and $Y=2\sum_{j=1}^n \norm{D_j}_{\Hi}^2$.
\end{proof}

 \section{Decoupling inequalities for $U$-statistics}

Let $U_n$ be a $U$-statistic defined as in \eqref{eq:def_U_stats}. Suppose
for simplicity that $h_{\gri}=h$.
In general, its treatment is a difficult problem because the
contribution of
each random variable $\xi_i$ appears in several coordinates of the function
$h$ and the random field defined by
$X_{\gri}=h\pr{\xi_{\gri}}$ is not stationary. Moreover, the involved 
filtrations 
are not easy to deal with.
A much easier object is the decoupled $U$-statistic $U_n^{\dec}$ for
which the entries are replaced by independent copies of $\pr{\xi_i}_{i\geq 1}$.
More precisely, we define
\begin{equation}\label{eq:def_U_stat_decouplee}
 U_n^{\dec}\pr{\pr{h_{\gri}}
}:=\sum_{\gri\in\inc^m_n}h_{\gri}
 \pr{\xi_{\gri}^{\dec}},
\end{equation}
$\xi_{\gri}^{\dec}=\pr{\xi_{i_1}^{\pr{1}},\dots,\xi_{i_m}^{\pr{m}}}$, and 
  the sequences $\pr{\xi_{i}^{\pr{1}}}_{i\geq 1},\dots,
\pr{\xi_{i}^{\pr{m}}}_{i\geq 1}$, are mutually independent and have  
the same distribution as the sequence $\pr{\xi_i}_{i\geq1}$.

An important feature of decoupled $U$-statistics is that the tail
of the $U$-statistic can be controlled in terms of that of
its decoupled version.
\begin{Proposition}[\cite{MR1334173}]\label{prop:decoupling}
For each $m\geq 1$, there exists a  constant $K_m$ depending only on $m$
such that  if  $\pr{\xi_i}_{i\geq 1}$ is an i.i.d. sequence taking values in a
measurable space $\pr{S,\Sca}$,  $\pr{\B,\norm{\cdot}_{\B}}$  a separable
Banach space,
$h_{\gri}\colon S^m\to\B$  and let $
U_n\pr{\pr{h_{\gri}}}$ and   $
U_n^{\dec}\pr{\pr{h_{\gri}}}$ are defined respectively as in
\eqref{eq:def_U_stats} and \eqref{eq:def_U_stat_decouplee}, then the
following inequality holds for each positive $x$:
\begin{equation*}
 \PP\pr{\norm{U_n\pr{\pr{h_{\gri}}}}_{\B} >x}
 \leq K_m\PP\pr{K_m\norm{U_n^{\dec}\pr{\pr{h_{\gri}}}}_{\B} >x}.
\end{equation*}
\end{Proposition}
Since we are interested in maxima of $U$-statistics, we will need the
following consequence of Proposition~\ref{prop:decoupling}.
\begin{Corollary}\label{cor:decoupling_max}
 For each $m\geq 1$, there exists a constant $K_m$ depending only on $m$
such that  if  $\pr{\xi_i}_{i\geq 1}$ is an i.i.d.\ sequence taking values in a
measurable space $\pr{S,\Sca}$,  $\pr{\Hi,\scal{\cdot}{\cdot}}$  a
separable Hilbert space,
$h_{\gri}\colon S^m\to\Hi$  and let $
U_n\pr{\pr{h_{\gri}}}$ and   $
U_n^{\dec}\pr{\pr{h_{\gri}}}$ are defined respectively as in
\eqref{eq:def_U_stats} and \eqref{eq:def_U_stat_decouplee}, then the
following inequality holds for each positive $x$ and $N\geq m$:
\begin{equation*}
 \PP\pr{\max_{m\leq n\leq N}\norm{U_n\pr{\pr{h_{\gri}}}}_{\Hi} >x}
 \leq K_m\PP\pr{K_m\max_{m\leq n\leq
N}\norm{U_n^{\dec}\pr{\pr{h_{\gri}}}}_{\Hi} >x}.
\end{equation*}
\end{Corollary}
\begin{proof}
 We apply Proposition~\ref{prop:decoupling} to the following setting:
 we take $\B=\Hi^N$ endowed with the norm $\norm{\pr{y_n}_{n=1}^N}_{\B}:=
 \max_{1\leq n\leq N}\norm{y_n}_{\Hi}$
and for $\gri\in \inc^m$ consider
$\widetilde{h_{\gri}}\colon S^m\to \Hi^N$ defined as
\begin{equation*}
 \widetilde{h_{\gri}}\pr{s_1,\dots,s_m}
 =\pr{ H_{\gri}^{\pr{n}}\pr{s_1,\dots,s_m}   }_{n=1}^N,
\end{equation*}
where
\begin{equation*}
 H_{\gri}^{\pr{n}}\pr{s_1,\dots,s_m} =
 h_{i\gri}\pr{s_1,\dots,s_m}\ind{i_m\leq n}.
\end{equation*}
In this way,
\begin{equation*}
 \max_{m\leq n\leq N}\norm{U_n\pr{\pr{h_{\gri}}}}_{\Hi}
 =\norm{U_N\pr{\pr{\widetilde{h_{\gri}}}}   }_{\B}
\end{equation*}
and a similar equality holds for the decoupled versions of the $U$-statistics.
\end{proof}
\end{appendix}
\def\polhk\#1{\setbox0=\hbox{\#1}{{\o}oalign{\hidewidth
  \lower1.5ex\hbox{`}\hidewidth\crcr\unhbox0}}}\def\cprime{$'$}
  \def\polhk#1{\setbox0=\hbox{#1}{\ooalign{\hidewidth
  \lower1.5ex\hbox{`}\hidewidth\crcr\unhbox0}}} \def\cprime{$'$}

\end{document}